\documentclass{article}

\usepackage{amsmath,amsthm}
\usepackage{amssymb}

\usepackage{enumerate}
\usepackage{url}

\usepackage{tikz}
\usetikzlibrary{arrows.meta,positioning}
\usepackage{float}

\numberwithin{equation}{section}
\theoremstyle{plain}
\newtheorem{thm}{Theorem}[section]
\newtheorem{prop}[thm]{Proposition}
\newtheorem{cor}[thm]{Corollary}
\newtheorem{lem}[thm]{Lemma}

\newtheorem{ex}[thm]{Example}
\newtheorem{rem}[thm]{Remark}

\newcommand{\ds}{\displaystyle}

\newcommand{\Hp}{\ensuremath{H^2_+}}

\newcommand{\Hm}{\ensuremath{H^2_-}}

\newcommand{\Linf}{\ensuremath{L^\infty}}

\newcommand{\Hinf}{\ensuremath{H^\infty}}
\newcommand{\TT}{\ensuremath{\mathbb{T}}}
\newcommand{\DD}{\ensuremath{\mathbb{D}}}
\newcommand{\CC}{\ensuremath{\mathbb{C}}}

\renewcommand{\AA}{\ensuremath{\mathcal{A}}}

\newcommand{\CCC}{\ensuremath{\mathcal{C}}}
\newcommand{\OO}{\ensuremath{\mathcal{O}}}
\newcommand{\GG}{\ensuremath{\mathcal{G}}}

\newcommand{\RR}{\ensuremath{\mathcal{R}}}

\newcommand{\KK}{\ensuremath{\mathcal{K}}}

\newcommand{\NN}{\ensuremath{\mathcal{N}^+}}
\newcommand{\Cg}{\ensuremath{C_{\ol g}}}

\newcommand{\ol}[1]{\ensuremath{\overline{#1}}}
\newcommand{\annot}[2]{\underbrace{#1}_{\smash[b]{\scriptstyle #2}}}
\newcommand{\bs}{\backslash}
\newcommand{\zero}{\ensuremath{\{0\}}}
\newcommand{\spn}{\ensuremath{\operatorname{span}}}

\newcommand{\kktil}{\ensuremath{\tilde k^\theta_\lambda}}
\newcommand{\kk}{\ensuremath{k^\theta_\lambda}}

\newcommand{\kertg}{\ensuremath{\ker T_g}}

\newcommand{\kerth}{\ensuremath{\ker T_h}}

\begin{document}

\title{Maximal functions and multipliers for subkernels of Toeplitz operators}

\author{M. Cristina Câmara$^{1,}$\thanks{Corresponding author.},
C. Carteiro$^{1}$,
C. Diogo$^{1,2}$
}

\date{}

\maketitle

\begin{center}
$^{1}$Center for Mathematical Analysis, Geometry and Dynamical Systems, Department of Mathematics, Instituto Superior T\'ecnico, Universidade de Lisboa, Av. Rovisco Pais, 1049-001, Lisboa, Portugal\\
$^{2}$ISCTE - Lisbon University Institute, Av. das Forças Armadas, 1649-026, Lisboa, Portugal
\end{center}

\begingroup
\renewcommand\thefootnote{}
\footnotetext{Emails:
\url{cristina.camara@tecnico.ulisboa.pt} (M. Cristina Câmara); 
\url{carlos.carteiro@tecnico.ulisboa.pt} (C. Carteiro);
\url{cristina.diogo@iscte-iul.pt} (C. Diogo).
}
\endgroup

\begin{abstract}
We study the relations between maximal functions in a Toeplitz kernel and those in a subkernel of the same Toeplitz operator, as well as the question of how multipliers between Toeplitz kernels act on subkernels.
We use those relations to obtain model space representations, in particular isometric model space representations and Hayashi's representation, for some important classes of Toeplitz kernels.


\end{abstract}


\section{Introduction}
\label{sec:intro}

Let $g\in\Linf(\TT) =: \Linf$, where \TT\ is the unit circle. We denote by $T_g$ the Toeplitz operator with symbol $g$, defined on the Hardy space of the unit disk, $\Hp:=H^2(\DD)$, identified as usual with a subspace of $L^2(\TT) =: L^2$, by
\begin{equation}
\label{eq:def_Toeplitz_operator}
T_g f = P^+ gf \quad , \quad f \in \Hp \ ,
\end{equation}
where $P^+$ denotes the orthogonal projection from $L^2$ onto \Hp.
We assume throughout the paper that $g \neq 0$.
Kernels of Toeplitz operators, also called Toeplitz kernels, have attracted a great interest, as they have fascinating properties, appear in numerous applications, and include important classes of analytic functions in the disk, such as model spaces.

Multipliers between Toeplitz kernels play an important role in understanding the structure of these spaces and in describing them.
In \cite{Cro94} Crofoot characterised the multipliers from a model space \emph{onto} another, describing in particular the isometric surjective multipliers between model spaces (see also \cite{Sa07}).
The more general question of characterising the multipliers from one model space \emph{into} another was studied by Fricain, Hartmann and Ross in \cite{FriHarRoss18}, see also \cite{FriRu18}, and, motivated by their work, C\^amara and Partington characterised in \cite{CP18_multipliers} the multipliers between two Toeplitz kernels in general.

Multipliers from a model space onto a Toeplitz kernel, allowing to describe the latter in the form
\begin{equation}
\label{eq:model_space_representation}
\kertg = w K_\theta \quad , \quad \text{with } \theta \text{ inner,}
\end{equation}
are particularly important.
The equality \eqref{eq:model_space_representation} is called a {\em model space representation} of \kertg;
if multiplication by $w$ is an isometry from $K_\theta$ onto \kertg, \eqref{eq:model_space_representation} is called an {\em isometric model space representation} (\cite{CCD25}).
Such a representation always exists.
Indeed Hayashi showed, in \cite{Ha86}, that the kernel of a Toeplitz operator, if non-trivial, can be written as
\begin{equation}
\label{eq:Hayashi_representation}
\kertg = u K_{z\gamma}
\end{equation}
where $u$ is outer, $\gamma$ is inner and $u$ multiplies the model space isometrically onto \kertg.
Hayashi's representation \eqref{eq:Hayashi_representation} is unique up to a unimodular constant.

The study of multipliers between Toeplitz kernels brings to light a strong connection between several distinct topics, such as Carleson measures, Smirnov spaces and maximal functions.
The latter were introduced in \cite{CP14}. They are elements of a Toeplitz kernel which cannot belong to any ``smaller'' Toeplitz kernel, i.e., functions $f^M \in \kertg$ such that, for any $h\in\Linf$,
\begin{equation}
f^M \in \kerth \implies \kertg \subset \kerth\ .
\end{equation}
We say that \kertg\ is the minimal Toeplitz kernel for $f^M$ and we write
\begin{equation}
\kertg = \KK_{min}(f^M)\ .
\end{equation}
It was shown in \cite{CP14} that in every non-trivial Toeplitz kernel there is a maximal function, which uniquely defines that kernel, and, conversely, for every $f\in\Hp$ there exists a minimal Toeplitz kernel to which $f$ belongs.

If $\AA$ is a Toeplitz kernel and $\AA \subset \kertg$, we say that \AA\ is a {\em subkernel of $T_g$}.
In this paper we study the relations between maximal functions in a Toeplitz kernel and maximal functions in subkernels of the same Toeplitz operator, and how multipliers between two Toeplitz kernels act on subkernels.
A crucial role is played by inner-outer factorisation of maximal functions and related factorisations of inner functions.
We use those relations to obtain model space representations, in particular isometric model space representations and Hayashi's representation, for some important classes of Toeplitz kernels, with an emphasis on Toeplitz kernels contained in model spaces.


\section{Toeplitz subkernels}
\label{sec:Toeplitz_subkernels}

The question when is a Toeplitz kernel contained in another Toeplitz kernel can be reduced to the question of characterising the multipliers between two Toeplitz kernels.
Note that, while multipliers from one model space into another must be in \Hp\ (\cite{FriHarRoss18}), multipliers between general Toeplitz kernels do not necessarily satisfy that condition;
they must however belong to the Smirnov class \NN\ (\cite{CP18_multipliers}).
They are characterised in Theorem \ref{thm:multiplicadores_into_kernels} below.
Let $\kertg\neq\zero$.
We say that $w\in\CCC(\kertg)$ whenever $|w|^2dm$ is a Carleson measure for \kertg, that is, $w\kertg \subset L^2$.

\begin{thm}[{\cite[Theorem 2.5]{CP18_multipliers}}]
\label{thm:multiplicadores_into_kernels}
Let $w\in\NN$.
Then the following are equivalent:
\begin{enumerate}[(i)]
\item $w \ker T_g \subset \ker T_h$;

\item $w\in\CCC(\kertg)$ and $\ds w h/g \in \ol\NN$;

\item $w\in\CCC(\kertg)$ and for some (and hence every) maximal function $\ds f^M_g$ in $\ker T_g$ we have $\ds w f^M_g \in \ker T_h$.
\end{enumerate}
\end{thm}

For $w=1$, this provides necessary and sufficient conditions for a Toeplitz kernel to be included in another Toeplitz kernel.

\begin{cor}[{\cite[Proposition 2.16]{CP18_multipliers}}]
\label{cor:kernel.subset.of.another.kernel}
$\ker T_g \subset \ker T_G \iff {G/g \in \ol\NN} \iff$ there exists a maximal function $f^M_g$ in $\ker T_g$ such that $f^M_g \in \ker T_G$.
\end{cor}

\begin{cor}[{\cite{CP18_multipliers}}]
\label{cor:subkernel_galpha}
$\ker T_g \subset \ker T_G \iff \ker T_g = \ker T_{G\alpha}\ $ for some inner function $\alpha \iff g = G\alpha \ol{O_1} / \ol{O_2}$ with $O_1,O_2 \in\Hinf$ outer.
\end{cor}

In this paper we will be particularly interested in Toeplitz kernels contained in model spaces.
A natural question regarding subkernels thus arises: which Toeplitz kernels contain some model space and, vice-versa, which Toeplitz kernels are contained in some model space?
We have the following.

\begin{prop}[{\cite{CP18_multipliers}}]
\label{prop:subkernel_modelspace}
Let $\theta$ be a non-constant inner function.
\begin{enumerate}[(i)]
\item If $K_\theta \subset \kertg$, then $\kertg$ is a model space.

\item If $\kertg \subset K_\theta$, then $\kertg = \ker T_{\ol\theta \alpha}\ $ for some inner function $\alpha$;
if $|g|=1$ then $g=\ol\theta \alpha$ for some inner function $\alpha$.
\end{enumerate}
\end{prop}

It follows from Proposition \ref{prop:subkernel_modelspace}(ii) that not all Toeplitz kernels can be contained in a model space.
Indeed, to every Toeplitz kernel one can associate a unimodular symbol (see \eqref{eq:unimod_symbol_from_max_func}).
If $\kertg \subset K_\theta$ with $|g|=1$, then we must have that $\theta g = \alpha$ for some inner function $\alpha$ and this may not be possible, as for instance if $g=\ol z^{3/2}$ ($\kertg = \spn \{ (1+z)^{1/2} \}$).
In particular, there are finite dimensional Toeplitz kernels that are not contained in any model space.

We can however, characterise some classes of Toeplitz operators whose kernels are contained in model spaces.
This is the case of Toeplitz operators with symbol in $\ol\Hinf$ (whose kernels are themselves model spaces) and all Toeplitz operators with rational symbols, as we show next.

\begin{prop}
\label{prop:rational_symbol_contained_model_space}
Let \RR\ denote the space of all rational functions without poles on \TT\ and let $g\in\RR$.
Then $\kertg\subset K_\theta$ for some inner function $\theta$.
\end{prop}

\begin{proof}
First note that, if $g$ has $n_\TT$ zeroes on \TT\ (counting multiplicities), we can write
\begin{equation}
\kertg = \ker T_{\frac{P}{Q}z^{n_\TT}}
\end{equation}
where $P$ and $Q$ are relatively prime polynomials without zeroes on \TT\ by \cite[Corollary 6.3]{CP24}.
Let us write
\begin{equation}
P = P_- P_+ \quad , \quad Q = Q_- Q_+ \ ,
\end{equation}
where $P_-,Q_-$ only have zeroes in \DD\ and $P_+,Q_+$ only have zeroes in the exterior of \DD, denoted $\DD_-$.
Let
\begin{equation}
\label{eq:n_def}
n = \deg Q_- - \deg P_- \ ;
\end{equation}
thus,
\begin{equation}
\frac{P}{Q}z^{n_\TT} = \annot{
\left( \frac{P_-}{Q_-}z^n \right)}{\in\GG\ol\Hinf}
z^{n_\TT - n} \annot{
\frac{P_+}{Q_+}}{\in\GG\Hinf},
\end{equation}
where we denote by $\GG\AA$ the group of invertible elements in a unital algebra \AA.

If $n_\TT - n \geq 0$, then $\kertg = \ker T_{\frac{P}{Q}z^{n_\TT}} = \zero$;
if $n_\TT - n < 0$, then
\begin{equation}
\kertg = \ker T_{\frac{P}{Q}z^{n_\TT}} = \ker T_{\frac{P_+}{Q_+}z^{n_\TT -n}} = \frac{Q_+}{P_+} K_{z^{n - n_\TT}} \neq\zero.
\end{equation}
Now, $P_+$ and $Q_+$ are products of factors of the form $z - 1/ \ol\lambda$, with $\lambda \in \DD$, and
\begin{equation}
z-\frac{1}{\ol\lambda} = - \frac{1}{\ol\lambda} \ 
\annot{\frac{1-\ol\lambda z}{z-\lambda}}{\ol{B_\lambda}} \ 
\annot{\frac{z-\lambda}{z}}{\in\GG\ol\Hinf} \ 
z \ ,
\end{equation}
where we denote by $B_\lambda$ the Blaschke factor
\begin{equation}
B_\lambda = \frac{z-\lambda}{1 - \ol\lambda z}\ .
\end{equation}

Thus, there are finite Blaschke products $B_1$ and $B_2$ such that 
\begin{align}
P_+ &= \ol{B_1} z^{n_1} \ol{h_1} \quad , \quad h_1\in\GG\Hinf\ , \label{eq:B1_def}\\
Q_+ &= \ol{B_2} z^{n_2} \ol{h_2} \quad , \quad h_2\in\GG\Hinf\ ,
\end{align}
where $n_1$ and $n_2$ are the number of zeroes of $P_+$ and $Q_+$ (in $\DD_-$), respectively.
Therefore
\begin{equation}
\kertg = \ker T_{\frac{P_+}{Q_+}z^{n_\TT -n}} = \ker T_{\ol{B_1}B_2 z^N}\ ,
\end{equation}
with
\begin{equation}
\label{eq:N_def}
N = n_\TT - n + n_1 - n_2\ .
\end{equation}
If $N\geq 0$, then $\kertg\subset \ker T_{\ol{B_1}} = K_{B_1}$;
if $N<0$, then $\kertg\subset \ker T_{\ol{B_1}\ol z^{|N|}} = K_{B_1 z^{|N|}}$.
\end{proof}

\begin{rem}
It is easy to see that a function $f\in\Hp$ belongs to some model space if and only if $\ol f$ is of bounded type, i.e., $\ol f = f_1 / f_2$ with $f_1,f_2 \in \Hinf$ (\cite{Ne70}).
Rational functions are of bounded type, hence each rational function in \Hp\ belongs to a model space.
Proposition \ref{prop:rational_symbol_contained_model_space} states that, when $g\in\RR$, there exists a model space to which all functions in \kertg, which are rational functions, belong.
\end{rem}

The model spaces in the last paragraph of the proof of Proposition \ref{prop:rational_symbol_contained_model_space} are in fact minimal, as a consequence of the following propositions.
We say that $K_\theta$ is the {\em minimal model space containing $\kertg$} if $\kertg \subset K_\theta$ and for any inner function $\beta$ such that $\kertg \subset K_\beta$, we have $K_\theta \subset K_\beta$.

\begin{prop}
\label{prop:minimal_model_space_condition}
Let $|g|=1$, $\kertg\neq\zero$.
Then the following are equivalent:
\begin{enumerate}[(i)]
\item There exists a minimal model space containing \kertg;

\item $g = \ol\theta \alpha$ with $\theta,\alpha$ inner and such that $\gcd (\theta, \alpha) = 1$.
\end{enumerate}
In that case, $K_\theta$ is the minimal model space containing \kertg.
\end{prop}

\begin{proof}
If (ii) holds then $\kertg \subset K_\theta$ and, if $\kertg \subset K_\beta$ for any other inner function $\beta$, then, by Proposition \ref{prop:subkernel_modelspace}(ii)
\begin{equation}
g = \ol\theta \alpha = \ol\beta \gamma
\end{equation}
for some inner function $\gamma$, so that $\theta \gamma = \alpha \beta$.
Since $\gcd (\theta, \alpha) = 1$, it follows that $\theta \preceq\beta$, i.e. $\theta\ol\beta\in\Hinf$, and $K_\theta \subset K_\beta$.

Conversely, if (i) holds, let $K_\theta$ be the minimal model space containing $\ker T_g$. Then, for any inner function $\beta$,
\begin{equation}
\label{eq:proof:minimal_model_space}
\kertg \subset K_\beta \implies K_\theta \subset K_\beta\ .
\end{equation}
Since $g= \ol\theta \alpha$ for some inner function $\alpha$, if $\gcd (\theta, \alpha) = \delta \notin\CC$, then $\kertg = \ker T_{(\ol\theta \delta) (\alpha \ol\delta)}$ and $\kertg \subset K_{\theta \ol\delta} \subsetneq K_\theta$, contradicting \eqref{eq:proof:minimal_model_space}, so we must have $\gcd (\theta, \alpha) = 1$.
\end{proof}

\begin{prop}
If $\zero \neq \kertg \subset K_\beta$ for some inner function $\beta$, then there exists a minimal model space $K_\theta$ containing \kertg.
If $g=\ol\beta \gamma$ with $\gamma$ inner, then $K_\theta = K_{\beta \ol\delta}$ where $\delta = \gcd (\beta, \gamma)$.
\end{prop}

\begin{proof}
If $\kertg \subset K_\beta$ then $g=\ol\beta \gamma = (\ol\beta \delta) (\gamma \ol\delta)$ where $\gcd (\beta \ol\delta, \gamma \ol\delta) = 1$.
\end{proof}

\begin{cor}
Let $g\in\RR$, with $n_\TT$ zeroes on \TT\ (counting multiplicities) and let $n>n_\TT$, where $n$ is defined as in \eqref{eq:n_def}.
Let moreover $B_1$ and $N$ be defined as in \eqref{eq:B1_def} and \eqref{eq:N_def}, respectively.
Then,
\begin{enumerate}[(i)]
\item if $N\geq 0$, $K_{B_1}$ is the minimal model space containing \kertg;

\item if $N<0$, $K_{B_1 z^{|N|}}$ is the minimal model space containing \kertg.
\end{enumerate}
\end{cor}

\begin{proof}
By construction, $B_1$ and $B_2$ do not have common zeroes and neither $B_1$ or $B_2$ vanish at $0$, so, if $N\geq 0$, we have
\begin{equation}
\kertg = \ker T_{\ol{B_1} ( B_2 z^N )} \text{ with } \gcd (B_1, B_2 z^N ) = 1
\end{equation}
and, if $N<0$,
\begin{equation}
\kertg = \ker T_{ \ol{ B_1 z^{|N|} } B_2 } \text{ with } \gcd (B_1 z^{|N|} , B_2 ) = 1.
\end{equation}
\end{proof}

We illustrate this result in the following example.

\begin{ex}
\label{ex:rational_symbol}
We have $K = \ker T_\frac{z-2}{z^2 (z-3) (z-4)} \subset K_{z^3 B_{\frac{1}{2}}}$.
To verify this inclusion, let $\phi_+ \in K$, i.e., $\phi_+\in\Hp$ and
\begin{equation}
\frac{z-2}{z^2 (z-3) (z-4)} \phi_+ = \phi_- \in \Hm \quad (\Hm := L^2 \ominus \Hp)\ .
\end{equation}
Then
\begin{equation}
\frac{1 - \frac{1}{2} z}{z- \frac{1}{2}} \ \frac{1}{z^3} \ \phi_+ =
\annot{- \frac{1}{2} \frac{(z-3) (z-4)}{z (z-\frac{1}{2})}}{\in\ol\Hinf}
\annot{\left( \frac{z-2}{z^2 (z-3) (z-4)} \phi_+ \right)}{\phi_- \in \Hm}
\in\Hm\ .
\end{equation}
So indeed $\phi_+ \in \ker T_{\ol z^3 \ol{B_{\frac{1}{2}}}} = K_{z^3 B_{\frac{1}{2}}}$.
We have that $K_{z^3 B_{\frac{1}{2}}}$ is the minimal model space containing $K$.
\end{ex}


\section{Maximal functions in subkernels}
\label{sec:maximal_functions_subkernels}

We now consider the relations between maximal functions in a Toeplitz kernel and in a subkernel of the same Toeplitz operator.

In every nontrivial Toeplitz kernel $\kertg$ one can find a maximal function $f^M$ (which is not unique) such that $\kertg = \KK_{min}(f^M)$.
If $f^M = IO$ is an inner-outer factorisation of that maximal function, with $I$ inner and $O\in\Hp$ outer (a notation that we will keep throughout the paper), then we can associate a unimodular symbol to \kertg\ (\cite{CP14}):
\begin{equation}
\label{eq:unimod_symbol_from_max_func}
\kertg = \ker T_{\ol{zI}\frac{\ol O}{O}}\ .
\end{equation}
Thus every maximal function in a Toeplitz kernel uniquely determines that kernel.
Maximal functions are characterised by the following necessary and sufficient condition.

\begin{prop}[{\cite[Theorem 2.2]{CP18_multipliers}}]
\label{prop:max_func_characterization}
$f^M \in \Hp$ is a maximal  function in $\ker T_g$ if and only if, for some $\OO\in\Hp$ outer, we have
\begin{equation}
\label{eq:max_func_characterization}
g f^M = \ol z \ol \OO\ .
\end{equation}
\end{prop}

We can express this result in terms of a conjugation (\cite{GarMashRoss16}), if $|g|=1$.
Then one can define a natural conjugation on \kertg\ (\cite{DyPlaPtak22})
\begin{equation}
\label{eq:def_conj_Cg}
\Cg f = \ol g \ol z \ol f \quad , \quad \text{for } f \in \kertg\ ,
\end{equation}
and the necessary and sufficient condition \eqref{eq:max_func_characterization} can be reformulated as follows.

\begin{prop}[{\cite[Theorem 4.1]{CCD25}}]
\label{prop:max_func_charact_with_conj}
If $|g|=1$, then $f^M\in\Hp$ is a maximal function in $\kertg$ if and only if $\Cg f^M$ is outer in \Hp.
\end{prop}

Moreover, we have the following.

\begin{prop}[{\cite[Proposition 4.3]{CCD25}}]
\label{prop:conj_max_func_is_Ofactor}
Let $|g|=1$.
If $f^M$ is a maximal function in $\ker T_g$, then $\Cg f^M$ is the outer factor in an inner-outer factorisation of $f^M$.
\end{prop}

In particular, if $g=\ol\theta$, where $\theta$ is an inner function, and for any $\lambda\in\DD$, denoting 
\begin{equation}
\kktil = \frac{\theta - \theta(\lambda)}{z - \lambda}
\quad \text{and} \quad
\kk = \frac{1 - \ol{\theta(\lambda)}\theta}{1 - \ol\lambda z} = C_\theta \kktil \ ,
\end{equation}
we have that \kktil\ is a maximal function in $K_\theta$ and
\begin{equation}
\label{eq:ktil_IO_fact}
\kktil = \frac{\kktil}{\kk} \ \kk\
\end{equation}
is an inner-outer factorisation of \kktil, with $\kktil / \kk$ inner and $\kk\in\GG\Hinf$ outer.

We now study the question of how to obtain maximal functions in a subkernel of $T_g$, given a maximal function in \kertg.

The following is a generalisation of Theorem 6.9 in \cite{CMP16}.

\begin{prop}
\label{prop:max_func_subkernel_times_u+_gives_max_func_big_kernel}
Let $u\in\Hinf\bs\zero$ and let $f^M$ be a maximal function in $\ker T_{G u}$.
Then $u f^M$ is a maximal function in $\ker T_G$.
\end{prop}

\begin{proof}
By Proposition \ref{prop:max_func_characterization}, $G u f^M = \ol z \ol\OO$ with $\OO\in\Hp$ outer, so $G (u f^M) = \ol z \ol\OO$ and $u f^M$ is a maximal function in $\ker T_G$.
\end{proof}

\begin{cor}[{\cite{CMP16}}]
\label{cor:max_func_subkernel}
Let $\zero \neq \kertg \subset \ker T_G$, with $\kertg = \ker T_{G \alpha}$ where $\alpha$ is inner.
If $f^M$ is a maximal function in \kertg\ then $F^M = \alpha f^M$ is a maximal function in $\ker T_G$.
Conversely, if $F^M \in \alpha \Hp$ is a maximal function in $\ker T_G$, then $f^M = \ol\alpha F^M$ is a maximal function in $\kertg = \ker T_{G \alpha}$.
\end{cor}

It follows from Corollary \ref{cor:max_func_subkernel} that the question of finding  maximal functions in a subkernel of $T_G$, $\ker T_{G \alpha}$, is equivalent to finding maximal functions in $\ker T_G$ whose inner factor can be divided by $\alpha$.

Note that, if $\ker T_{G\alpha} \neq \zero$, then there exists a maximal function $F^M$ in $\ker T_G$, with inner-outer factorisation $F^M = IO$, where the inner factor $I$ is such that $\ker T_{\ol z \alpha \ol I} \neq\zero$ (cf. Corollary \ref{cor:max_func_subkernel}).
We will focus here on the case where $\alpha = B$ is a finite Blaschke product, with $\dim K_B =N \geq 1$.
Then
\begin{equation}
\ker T_{GB} \neq\zero \iff \dim \ker T_G >N
\end{equation}
(\cite{CMP16}) and there exists a maximal function in $\ker T_G$, as above, with inner-outer factorisation
\begin{equation}
F^M = IO \quad , \quad \text{where } \dim K_I \geq N.
\end{equation}

This is an important case which is related, for instance, with the interpolation problem of describing the subspace of all functions in $\ker T_G$ that have zeroes of a prescribed order at certain points (which is of the form $B \ker T_{GB}$ where $B$ is a finite Blaschke product),
with invariance properties regarding multiplication operators and orthogonal decompositions of Toeplitz kernels (\cite{CaKlisPtak24_XYinv}),
and with kernels of asymmetric truncated Toeplitz operators (\cite{CP24}).

We start by considering the case where $N=1$.

\begin{prop}
\label{prop:max_func_vanishing}
Let $\dim \ker T_G >1$.
If $F^M$ is a maximal function in $\ker T_G$, with inner-outer factorisation $F^M = IO$, where $I$ is a non-constant inner function and $O$ is outer, then
\begin{equation}
\label{eq:max_func_with_zero}
F^M_\lambda = (z-\lambda) \tilde k^I_\lambda O = B_\lambda \tilde k^I_\lambda (1-\ol\lambda z) O
\end{equation}
is a maximal function in $\ker T_G$, for any $\lambda\in\DD$ .
\end{prop}

\begin{proof}
Assume that $|G|=1$.
Since $F^M = IO$ is a maximal function in $\ker T_G$, we have that $C_{\ol G} (IO) = cO$ for some constant $c$ with $|c|=1$, so
\begin{equation}
C_{\ol G} [ (z-\lambda) \tilde k^I_\lambda O ] = C_{\ol G} [ (I - I(\lambda)) O ] = cO - \ol{I(\lambda)} c I O = (1 - \ol{I(\lambda)} I) cO
\end{equation}
where the right-hand side represents an outer function.
Thus $F^M_\lambda$ is maximal in $\ker T_G$ by Proposition \ref{prop:max_func_charact_with_conj}.
\end{proof}

\begin{cor}
With the same assumptions as in Proposition \ref{prop:max_func_vanishing},  
\begin{equation}
f^M_\lambda = \tilde k^I_\lambda (1-\ol\lambda z) O
\end{equation}
is a maximal function in $\ker T_{G B_\lambda}$.
For $\lambda = 0$, we have that
\begin{equation}
F^M_0 = (I - I(0)) O = z \tilde k_0^I O
\end{equation}
is a maximal function in $\ker T_G$ vanishing at 0 and
\begin{equation}
f^M_0 = \tilde k_0^I O
\end{equation}
is a maximal function in $\ker T_{Gz}$.
\end{cor}

The family of maximal functions \eqref{eq:max_func_with_zero} can be extended by taking an alternative approach, based on a factorisation of the inner factor of $F^M$, which will allow us to define maximal functions whose inner factor is divisible by a finite Blaschke product.

Any non-constant inner function $\alpha$ can be factorised as
\begin{equation}
\label{eq:mod_max_f_fact_alpha_lambda_mu}
\alpha = B_\lambda\ \frac{\tilde k^\alpha_\mu}{k^\alpha_\mu}\ \frac{k^\alpha_\mu (1-\ol\lambda z)}{\ol{k^\alpha_\mu} (1-\lambda\ol z)} \quad, \quad \lambda,\mu\in\DD\ ,
\end{equation}
where
\begin{equation}
\alpha_1 = \frac{\tilde k^\alpha_\mu}{k^\alpha_\mu}
\end{equation}
is an inner function (more precisely, $\alpha_1$ is the inner factor of $\tilde k^\alpha_\mu$) and $k^\alpha_\mu (1-\ol\lambda z) \in \GG\Hinf$;
\eqref{eq:mod_max_f_fact_alpha_lambda_mu} is called a {\em modified maximal function factorisation of $\alpha$} (\cite{CCD25}).

If $F^M = IO$ is the inner-outer factorisation of a maximal function $F^M$ in $\ker T_G$ and $\alpha$ is a non-constant inner function such that $\alpha \preceq I$, then we can write, for any $\lambda_1,\mu_1\in\DD$,
\begin{align}
\ker T_G & = \ker T_{\ol z \ol I \frac{\ol O}{O}} = \ker T_{\ol z \ol{(\frac{I}{\alpha})} \ol\alpha \frac{\ol O}{O}} = 
\ker T_{
\ol z \ol{(\frac{I}{\alpha})} \ol{B_{\lambda_1}} \ol{\left(
\frac{\tilde k^\alpha_{\mu_1}}{k^\alpha_{\mu_1}}
\right)} \frac{\ol{k^\alpha_{\mu_1}} (1-{\lambda_1} \ol z)}{k^\alpha_{\mu_1} (1-\ol{\lambda_1} z)} \frac{\ol O}{O}
} \notag \\
 & = k^\alpha_{\mu_1} (1-\ol{\lambda_1} z) \ker T_{ \ol z \ol{(\frac{I}{\alpha})} \ol{B_{\lambda_1}} \ol{\alpha_1} \frac{\ol O}{O} }\ .
\label{eq:big_equation}
\end{align}

Since, by \eqref{eq:unimod_symbol_from_max_func}, $\frac{I}{\alpha} B_{\lambda_1} \alpha_1 O$ is a maximal function in the kernel on the right-hand side of \eqref{eq:big_equation}, it follows from \cite[Theorem 3.2]{CP18_multipliers} that
\begin{equation}
\label{eq:big_family_max_f}
F^M_1 = k^\alpha_{\mu_1} (1-\ol{\lambda_1} z) \frac{I}{\alpha} B_{\lambda_1} \alpha_1 O = B_{\lambda_1} \frac{I}{\alpha} \tilde k^\alpha_{\mu_1} (1-\ol{\lambda_1} z) O
\end{equation}
is a maximal function in $\ker T_G$ vanishing at $\lambda_1$.
Note that, for $\alpha = I$ and $\mu_1 = \lambda_1$, we recover \eqref{eq:max_func_with_zero}.

We have thus shown the following.

\begin{prop}
\label{prop:max_func_subkernel_GB_1}
If $F^M = IO$ is a maximal function in $\ker T_G$ and $\alpha \preceq I$, where $\alpha$ is a non-constant inner function, then \eqref{eq:big_family_max_f} is a maximal function in $\ker T_G$, for any $\mu_1,\lambda_1\in\DD$, and
\begin{equation}
\frac{I}{\alpha} \tilde k^\alpha_{\mu_1} (1-\ol{\lambda_1} z) O
\end{equation}
is a maximal function in $\ker T_{G B_{\lambda_1}}$.
\end{prop}

Let now $I=\alpha$ in \eqref{eq:big_family_max_f} and define
\begin{equation}
I_1 = \frac{\tilde k^I_{\mu_1}}{k^I_{\mu_1}} \quad , \quad O_1 = k^I_{\mu_1} (1-\ol{\lambda_1} z) O
\end{equation}
so that we have the inner-outer factorisation
\begin{equation}
\label{eq:M}
F_1^M = B_{\lambda_1} I_1 O_1\ .
\end{equation}

If $\dim K_I \geq 2$, by Lemma \ref{lem:dim_K_In} below, we have that $\dim K_{I_1} \geq 1$ and so, applying Proposition \ref{prop:max_func_subkernel_GB_1} with $\alpha = I_1, \frac{I}{\alpha} = B_{\lambda_1}$ and $O$ replaced by $O_1$, we get from \eqref{eq:M} that, for any $\mu_2,\lambda_2\in\DD$,
\begin{equation}
F^M_2 = B_{\lambda_1} B_{\lambda_2} \ 
\annot{ \frac{\tilde k_{\mu_2}^{I_1}}{k_{\mu_2}^{I_1}} }{I_2} \ 
\annot{ k_{\mu_2}^{I_1} (1-\ol{\lambda_2} z) O_1 }{O_2}
= B_{\lambda_1} B_{\lambda_2} I_2 O_2
\end{equation}
is a maximal function in $\ker T_G$ whose inner factor is divisible by $B_{\lambda_1} B_{\lambda_2}$.

By repeating this reasoning, and taking Lemma \ref{lem:dim_K_In} into account, we get the following.

\begin{thm}
\label{thm:max_func_subkernel_gB}
Let $\dim \ker T_G > N$ and $B = B_{\lambda_1} B_{\lambda_2} \cdots B_{\lambda_N}$, with $\lambda_j\in\DD$ for $j=1,2,\dots, N$.

If $F^M = IO$ is the inner-outer factorisation of a maximal function $F^M$  in $\ker T_G$, with $\dim K_I \geq N$, then, for $I_N$ and $O_N$ defined by 
\begin{align}
I_0 	&= I \quad \qquad , \quad O_0 = O\ , \label{eq:def_In}\\
I_n &= \frac{\tilde k^{I_{n-1}}_{\lambda_n}}{k^{I_{n-1}}_{\lambda_n}}
\quad , \quad 
O_n = k^{I_{n-1}}_{\lambda_n} (1 - \ol{\lambda_n} z) O_{n-1}\ , \label{eq:def_On}
\end{align}
for $1\leq n \leq N$, we have that
\begin{equation}
\label{eq:max_func_vanishing_B}
F^M_B = B I_N O_N
\end{equation}
is also a maximal function in $\ker T_G$ and
\begin{equation}
f^M_B = I_N O_N
\end{equation}
is a maximal function in $\ker T_{G B}$.
\end{thm}

\begin{lem}
\label{lem:dim_K_In}
Let $I$ be an inner function with $\dim K_I \geq N$ and let $I_n$ be defined, for $1 \leq n \leq N$, by \eqref{eq:def_In}-\eqref{eq:def_On}.
Then
\begin{equation}
\dim K_{I_{N-1}} \geq 1.
\end{equation}
\end{lem}

\begin{proof}
For any inner function $I$ with $\dim K_I \geq N$ and any $\mu\in\DD$, $\tilde k^I_\mu$ is a maximal function in $K_I$ with inner-outer factorisation $\tilde k^I_\mu = \frac{\tilde k^I_\mu}{k^I_\mu} k^I_\mu = I_1 k^I_\mu$.
So, by \eqref{eq:unimod_symbol_from_max_func}, $K_I = \ker T_{\ol z \ol{I_1} \frac{\ol{k^I_\mu}}{k^I_\mu}} = k^I_\mu \ker T_{\ol z \ol{I_1}}$.
It follows that $\dim K_{I_1} \geq N-1$.
By repeating this argument we obtain the result.
\end{proof}

\begin{rem}
Theorem \ref{thm:max_func_subkernel_gB} can also be used to obtain a maximal function in $\ker T_{G \alpha}$ where $\alpha$ is not a finite Blaschke product, but is related to $I$ by
\begin{equation}
I = \alpha \gamma \ol B
\end{equation}
where $\gamma$ is an inner function dividing $I$ and $B$ is a finite Blaschke product.
Then $\ker T_{G \alpha} = \ker T_{g B}$ with $g = G I \ol\gamma$ and, if $F^M_G = IO$ is the inner-outer factorisation of a maximal function in $\ker T_G$, then, by Proposition \ref{prop:max_func_charact_with_conj}, $F^M_g = \gamma O$ is a maximal function in $\ker T_g$, to which we can apply Proposition \ref{prop:max_func_subkernel_times_u+_gives_max_func_big_kernel} in order to obtain a maximal function in $\ker T_{G \alpha}$.
\end{rem}

\begin{ex}
We illustrate the results of Theorem  \ref{thm:max_func_subkernel_gB} in the case where $\ker T_G = K_\theta$, with $\theta$ inner (non-constant).
We have that, for any $\lambda\in\DD$, \kktil\ is a maximal function in $K_\theta$, with inner-outer factorisation given by \eqref{eq:ktil_IO_fact}.
So, since $\dim K_\theta >1$, by Proposition \ref{prop:max_func_vanishing}, a maximal function in $K_\theta$ vanishing at a given point $\lambda_1 \in \DD$ is 
\begin{align}
F^M_{\lambda_1} & = B_{\lambda_1} \frac{\tilde k^I_{\lambda_1}}{k^I_{\lambda_1}} (1- \ol{ I(\lambda_1) } I ) \kk = ( I - I(\lambda_1) ) \kk \nonumber \\
& = \frac{1}{\kk(\lambda_1)} [ \kk(\lambda_1) \kktil - \kktil(\lambda_1) \kk ]\ ,
\end{align}
where
\begin{equation}
I= \frac{\kktil}{\kk}\ .
\end{equation}

A maximal function in $\ker T_{\ol\theta B_{\lambda_1}}$ is thus
\begin{equation}
f^M_{\lambda_1} = [ \kk(\lambda_1) \kktil - \kktil(\lambda_1) \kk ] \ \ol{ B_{\lambda_1} }\ .
\end{equation}
In particular, a maximal function in $\ker T_{\ol\theta z}$ is
\begin{equation}
f^M_0 =  [ k^\theta_0(0) \tilde k^\theta_0 - \tilde k^\theta_0(0) k^\theta_0 ] \ \ol z\ .
\end{equation}

Assuming that $\dim K_\theta >N$, a maximal function in $K_\theta$ with a zero of order at least $N$ at the point $0$ is, by \eqref{eq:max_func_vanishing_B},
\begin{equation}
F^M_{z^N} = z^N I_N k_0^\theta k_0^{I_0} \cdots k_0^{I_{N-1}}
\end{equation}
where $I_n$ , for $n=1,\dots,N$, is defined by
\begin{equation}
I_0 = \frac{\tilde k_0^\theta}{k_0^\theta} \quad , \quad I_n = \frac{\tilde k_0^{I_{n-1}}}{k_0^{I_{n-1}}}\ ;
\end{equation}
therefore, $\ol z^N F^M_{z^N}$ is a maximal function in $\ker T_{\ol\theta z^N}$.
\end{ex}


\section{Multipliers on subkernels}
\label{sec:multipliers_subkernels}

We now study the question of how multipliers between two Toeplitz kernels act on subkernels and, in particular, whether they map subkernels onto subkernels.

If $\kertg \subset \ker T_G$ and $w$ is a multiplier from $\ker T_G$ \emph{into} $\ker T_H$, i.e., $w \ker T_G \subset \ker T_H$, then we have that $w \kertg \subset \ker T_H $, but $w \kertg$ may not be a subkernel of $T_H$.
As a concrete example, if $\alpha$ and $\theta$ are inner functions, $\alpha$ multiplies $K_\theta$ into $K_{\theta\alpha}$, but $\alpha K_\theta$ is not a Toeplitz kernel (\cite{CP14,CP24}).
However, as we show next, the situation is different if $w$ multiplies $\ker T_G$ \emph{onto} $\ker T_H$.

We have the following characterisation of multipliers from one Toeplitz kernel onto another.
Note that such multipliers must belong to the Smirnov class \NN, along with their inverses (\cite{CP18_multipliers}).

\begin{thm}[{\cite{CP18_multipliers}}]
\label{thm:multiplicadores_onto_kernels}
Let $w^{\pm 1} \in \NN$.
Then the following are equivalent:
\begin{enumerate}[(i)]
\item $w \ker T_g = \ker T_h$;

\item $w \in \CCC (\ker T_g)$, $w^{-1} \in \CCC (\ker T_h)$ and $\frac{h}{g} = \frac{\ol w}{w} \frac{\ol O_1}{\ol O_2}$ with $O_1,O_2 \in \Hp$, outer;

\item $w \in \CCC (\ker T_g)$, $w^{-1} \in \CCC (\ker T_h)$ and, for some (and hence every) maximal function $f^M_g$ in $\ker T_g$, we have that $w f^M_g$ is a maximal function in $\ker T_h$.
\end{enumerate}

Moreover,
\begin{equation}
\label{eq:simbolo_vindo_de_multiplicador}
\ker T_h = w \ker T_g \implies \ker T_h = \ker T_{g\frac{\ol w}{w}}\ .
\end{equation}

\end{thm}

For $w=1$, we obtain from Theorem \ref{thm:multiplicadores_onto_kernels} necessary and  sufficient conditions for two Toeplitz operators to have the same kernel.

\begin{cor}[{\cite[Corollary 2.19]{CP18_multipliers}}]
\label{cor:kernels_iguais}
$\ker T_g = \ker T_h \iff \frac{h}{g} = \frac{\ol O_1}{\ol O_2}$ for some $O_1,O_2\in\Hp$ outer.

If $g\in\GG\Linf$ then $\ker T_g = \ker T_h \iff h = g \ol{O}$ with $O\in\Hinf$, outer.
In particular, if $|g| = |h| = 1$, then 
\begin{equation}
\kertg = \kerth \iff h = \lambda g \text{ with } \lambda\in\CC, |\lambda| =1.
\end{equation}
\end{cor}

It follows from Theorem \ref{thm:multiplicadores_onto_kernels}, Corollary \ref{cor:kernels_iguais} and \eqref{eq:simbolo_vindo_de_multiplicador} that the question when is $w\kertg$ a Toeplitz kernel, assuming $w^{\pm 1} \in\NN$, is reduced to the question when is $\ker T_{g\frac{\ol w}{w}} = w \kertg$.
For subkernels of a Toeplitz operator, we have the following.

\begin{prop}
\label{prop:multiplier_between_subkernels}
Let $\zero\neq \kertg \subset \ker T_G$ and suppose that $w \ker T_G = \ker T_H$.
Then $w \kertg$ is a subkernel of $T_H$ and we have
\begin{equation}
\label{eq:multiplier_between_subkernels}
w \kertg = \ker T_{g \frac{\ol w}{w}}\ .
\end{equation}
\end{prop}

\begin{proof}
From the assumptions and Theorem \ref{thm:multiplicadores_onto_kernels} we have that $w^{\pm 1} \in\NN$ and $w \kertg \subset w \ker T_G \subset L^2$.
On the other hand, $\ker T_H = \ker T_{G \frac{\ol w}{w}}$, by \eqref{eq:simbolo_vindo_de_multiplicador}, and $\kertg = \ker T_{G\alpha}$ for some inner function $\alpha$ by Corollary \ref{cor:subkernel_galpha}, so
\begin{equation}
w^{-1} \ker T_{g \frac{\ol w}{w}} = w^{-1} \ker T_{G \alpha \frac{\ol w}{w}} \subset w^{-1} \ker T_{G \frac{\ol w}{w}} = w^{-1} \ker T_H \subset L^2.
\end{equation}
Moreover, if $f^M_g$ is a maximal function in \kertg, then $\alpha f^M_g$ is a maximal function in $\ker T_G$, by Corollary \ref{cor:max_func_subkernel}, so $w\alpha f^M_g$ is a maximal function in $\ker T_H$, by Theorem \ref{thm:multiplicadores_onto_kernels}(iii).
Applying Corollary \ref{cor:max_func_subkernel} again, $w f^M_g$ is a maximal function in $\ker T_{H \alpha} = \ker T_{G \alpha \frac{\ol w}{w}} = \ker T_{g \frac{\ol w}{w}}$.
It follows from Theorem \ref{thm:multiplicadores_onto_kernels}(iii) that $w \kertg = \ker T_{g \frac{\ol w}{w}}$.
\end{proof}

Thus we see that a multiplier from a Toeplitz kernel onto another Toeplitz kernel maps, by multiplication, a subkernel of the first Toeplitz operator onto a subkernel of the second Toeplitz operator.

A simple but important consequence is the following.
Recall that
\begin{equation}
\kertg \subset \ker T_G \iff \kertg = \ker T_{G \alpha} \text{ for some inner function } \alpha.
\end{equation}

\begin{cor}
\label{cor:multiplier_between_subkernels_v2}
If $w \ker T_G = \ker T_H$ then $w \ker T_{G \alpha} = \ker T_{H \alpha}$, for any inner function $\alpha$.
\end{cor}

\begin{proof}
Since $\ker T_{G \alpha} \subset \ker T_G$, by \eqref{eq:multiplier_between_subkernels} and Proposition \ref{prop:multiplier_between_subkernels} we have that $w \ker T_{G \alpha} = \ker T_{G \alpha \frac{\ol w}{w}}$.
Since $\ker T_H = \ker T_{G \frac{\ol w}{w}}$ it follows from Corollary \ref{cor:kernels_iguais} that $H = G \frac{\ol w}{w} 	\frac{\ol{O_1}}{\ol{O_2}}$ where $O_1,O_2\in\Hp$ are outer, and therefore $\ker T_{H \alpha} = \ker T_{G \alpha
\frac{\ol w}{w} \frac{\ol{O_1}}{\ol{O_2}}} = \ker T_{G \alpha
\frac{\ol w}{w}} = w \ker T_{G \alpha}$.
\end{proof}

Corollary \ref{cor:multiplier_between_subkernels_v2} generalises a result by Crofoot on model spaces (\cite[Corollary 19]{Cro94}).

We also have:

\begin{thm}[{\cite[Theorem 10]{Cro94}}]
\label{thm:Crofoot_isom_multiplier}
Given any model space $K_{\theta_1}$ there exists an isometric multiplier $m$ from $K_{\theta_1}$ onto $K_{\theta_2}$, where $\theta_2$ is an inner function, if and only if
\begin{equation}
\theta_2 = \frac{\theta_1 - \ol p}{1 - p \theta_1} \quad \text{ with } p\in\CC, |p|=1,
\end{equation}
(up to a unimodular constant) and in that case
\begin{equation}
m = m_p^{\theta_1} = \frac{\sqrt{1 - |p|^2}}{1 - p \theta_1}
\end{equation}
(up to a unimodular constant).
\end{thm}

In particular, we have that Hayashi's representation of a model space $K_\theta$ is
\begin{equation}
\label{eq:Hayashi_repres_modelspace}
K_\theta = \frac{1 - \ol{ \theta (0) } \theta}{\sqrt{1 - |\theta(0)|^2}} K_{\theta_H} = \frac{k_0^\theta}{\sqrt{1 - |\theta(0)|^2}} = K_{z \frac{\tilde k^\theta_0}{k^\theta_0}}
\end{equation}
with 
\begin{equation}
\theta_H = \frac{\theta - \theta(0)}{1 - \ol{ \theta (0) } \theta} = z \frac{\tilde k^\theta_0}{k^\theta_0}\ .
\end{equation}

The question of obtaining a representation of the form \eqref{eq:model_space_representation} for a Toeplitz kernel is closely related with that of determining maximal functions in $\ker T_G$. Indeed, by Theorem \ref{thm:multiplicadores_onto_kernels}, there is a strong connection between multipliers from one Toeplitz kernel onto another and maximal functions in those kernels.

Multipliers from a model space onto a Toeplitz kernel are often determined by, or equal to, the outer factor of a maximal function in that kernel.
For instance, if $w$ is a multiplier from a model space $K_\theta$ onto \kertg, i.e., $\kertg = w K_\theta$, and $\theta(0) =0$ (as in \eqref{eq:Hayashi_representation}), then, since $\theta \ol z$ is a maximal function in $K_\theta$, by Theorem \ref{thm:multiplicadores_onto_kernels} we have that $f^M = w \cdot (\theta \ol z)$ is a maximal function in \kertg, i.e., the multiplier $w$ is the outer factor of $f^M$.
The following result gives a necessary and sufficient condition for the outer factor of a maximal function in a Toeplitz kernel to be a multiplier from a model space onto that kernel.

\begin{prop}[{\cite{CCD25}}]
\label{prop:multiplier_from_maximal_function}
Let $f^M$ be a maximal function in \kertg\ and let $f^M = IO$ be an inner-outer factorisation.
Then $\kertg = O K_\theta$, with $\theta$ inner, if and only if $K_\theta = K_{zI}$ and
\begin{equation}
\label{eq:Carleson_cond_max_f}
O \in \CCC (K_\theta) \quad , \quad O^{-1} \in \CCC (\kertg).
\end{equation}
\end{prop}

\begin{proof}
If $\theta = zI$ and \eqref{eq:Carleson_cond_max_f} holds, then, since $I$ is a maximal function in $K_{zI}$ and $OI$ is a maximal function in \kertg, by Theorem \ref{thm:multiplicadores_onto_kernels} we have that $\kertg = O K_{zI}$.
Conversely, if $\kertg = O K_\theta$, then \eqref{eq:Carleson_cond_max_f} holds and $I$ must be a maximal function in $K_\theta$, so $K_\theta = K_{zI}$.
\end{proof}

From Proposition \ref{prop:multiplier_from_maximal_function} and Corollary \ref{cor:multiplier_between_subkernels_v2} we have the following.

\begin{cor}[{\cite[Corollary 3.17]{CCD25}}]
Let $f^M = IO$ be an inner-outer factorisation of a maximal function in \kertg, with $O\in\GG\Hinf$.
Then $\kertg = OK_{zI}$.
\end{cor}

\begin{cor}
\label{cor:multiplier_from_maximal_function_hz}
Let $f^M = IO$ be an inner-outer factorisation of a maximal function in $\ker T_G$, where $O$ satisfies \eqref{eq:Carleson_cond_max_f}.
Then 
\begin{equation}
\label{eq:multiplier_from_max_funct_hz}
\ker T_{Gz} = O K_I .
\end{equation}
\end{cor}

Using the relation
\begin{equation}
z= B_\lambda \frac{1-\ol\lambda z}{1-\lambda\ol z} \quad , \quad \lambda \in\DD,
\end{equation}
we also have:

\begin{cor}
\label{cor:multiplier_from_maximal_function_modified}
With the same assumptions as in Corollary \ref{cor:multiplier_from_maximal_function_hz}, we have that, for any $\lambda\in\DD$,
\begin{equation}
\label{eq:multiplier_from_maximal_function_modified}
\ker T_G = (1-\ol\lambda z) O K_{B_\lambda I}\ ,
\end{equation}
and
\begin{equation}
\label{eq:multiplier_from_maximal_function_modified_hB}
\ker T_{GB_\lambda} = O (1-\ol\lambda z) K_I\ .
\end{equation}
\end{cor}

Comparing \eqref{eq:multiplier_from_maximal_function_modified} and \eqref{eq:multiplier_from_maximal_function_modified_hB}, one may ask when does a multiplier which maps a Toeplitz kernel onto a model space, also map a subkernel of the same Toeplitz operator onto a model space.
By Corollary \ref{cor:multiplier_between_subkernels_v2}, if $\ker T_G = w K_\theta$, with $\theta$ inner, then
\begin{equation}
\label{eq:multiplier_Galpha_to_theta_alpha}
\ker T_{G \alpha} = w \ker T_{\ol\theta\alpha}
\end{equation}
for $\alpha$ inner;
$\ker T_{\ol\theta\alpha}$ is a model space if and only if $\alpha \preceq \theta$.

From \eqref{eq:multiplier_Galpha_to_theta_alpha} we see that, if $\ker T_G = w K_\theta$ and $\alpha$ is an inner function that does not divide $\theta$, then the question of obtaining a model space representation for a subkernel $\ker T_{G\alpha}$ reduces to obtaining a model space representation for $\ker T_{\ol\theta\alpha}$.
This is closely related with determining maximal functions in $K_\theta$ with an inner factor that can be divided by $\alpha$, discussed in Section \ref{sec:maximal_functions_subkernels}.


\section{Toeplitz kernels contained in model spaces}
\label{sec:Toeplitz_kernels_contained_in_model_spaces}

By Proposition \ref{prop:subkernel_modelspace}, Toeplitz kernels contained in some model space $K_\theta$ take the form $\ker T_{\ol\theta \alpha}$, for some inner function $\alpha$.
We focus here, as in Section \ref{sec:maximal_functions_subkernels}, on the case where $\alpha$ is a finite Blaschke product, i.e., on model space representations of Toeplitz kernels of the form $\ker T_{\ol\theta B}$ where $B$ is a finite Blashcke product.

\begin{thm}
\label{thm:model_space_repres_ker_theta_Blambda}
Let $\lambda\in\DD$ and let $\theta$ be an inner function with $\dim K_\theta >1$.
Then, for any $\mu\in\DD$,
\begin{enumerate}[(i)]
\item $\ker T_{\ol\theta B_\lambda} = (1-\ol\lambda z) k_\mu^\theta \, K_{\theta_\mu}$, with $\theta_\mu = \frac{\tilde k_\mu^\theta}{k_\mu^\theta}$;

\item $\ker T_{\ol\theta B_\lambda} = m_{\lambda,\theta}\, K_{\theta_\lambda}$, with $\theta_\lambda = \frac{\tilde k_\lambda^\theta}{k_\lambda^\theta}$, where
\begin{equation}
m_{\lambda,\theta} = \frac{(1-\ol\lambda z) k_\lambda^\theta}{\sqrt{1-|\theta(\lambda)|^2}}
\end{equation}
is an isometric multiplier from $K_{\theta_\lambda}$ onto $\ker T_{\ol\theta B_\lambda}$;

\item Hayashi's representation of $\ker T_{\ol\theta B_\lambda}$ is given by
\begin{equation}
\ker T_{\ol\theta B_\lambda} = m_{\lambda,\theta} \frac{k_0^{\theta_\lambda}}{\sqrt{1 - |\theta_\lambda(0)|^2}}\ K_{z (\theta_\lambda)_{{}_0}}\ .
\end{equation}
\end{enumerate}
\end{thm}

\begin{proof}
(i) From \eqref{eq:mod_max_f_fact_alpha_lambda_mu}, for any $\lambda,\mu\in\DD$ we have that
\begin{equation}
\ol\theta B_\lambda = \ol{\left( \frac{\tilde k_\mu^\theta}{k_\mu^\theta}    \right)} \frac{\ol{k_\mu^\theta} (1-\lambda\ol z)}{k_\mu^\theta (1 -\ol\lambda z)}
\end{equation}
where $k_\mu^\theta (1 -\ol\lambda z) \in \GG\Hinf$, so
\begin{equation}
\ker T_{\ol\theta B_\lambda} = k_\mu^\theta (1 -\ol\lambda z) K_{\theta_\mu}\ .
\end{equation}

(ii) For $\mu = \lambda$ we have that $\ker T_{\ol\theta B_\lambda} = k_\lambda^\theta (1 -\ol\lambda z) K_{\theta_\lambda}$.
Since, by Theorem \ref{thm:Crofoot_isom_multiplier}, $m_{\lambda,\theta}$ is an isometric multiplier from $K_{B_\lambda \theta_\lambda}$ onto $K_\theta$ and $K_{B_\lambda \theta_\lambda} \supset K_{\theta_\lambda}$, $K_\theta \supset \ker T_{\ol\theta B_\lambda}$, we have that $m_{\lambda,\theta}$ is an isometric multiplier from $K_{\theta_\lambda}$ onto $\ker T_{\ol\theta B_\lambda}$.

(iii) follows from \eqref{eq:Hayashi_repres_modelspace}.
\end{proof}

If $\dim K_\theta >2$, then $\dim K_{\theta_\mu}>1$ and we can apply Theorem \ref{thm:model_space_repres_ker_theta_Blambda} to $\ker T_{\ol{\theta_\mu} B_{\lambda_2}}$ for any $\lambda_2\in\DD$.
We thus get that, for any $\lambda_1,\lambda_2,\mu_1,\mu_2\in\DD$,
\begin{align}
\ker T_{\ol\theta B_{\lambda_1} B_{\lambda_2}} & = 
(1-\ol{\lambda_1} z) k_{\mu_1}^\theta \cdot \ker T_{\ol\theta_{\mu_1} B_{\lambda_2}} \nonumber\\
	& = (1-\ol{\lambda_1} z) (1-\ol{\lambda_2} z) k_{\mu_1}^\theta k_{\mu_2}^{\theta_{\mu_1}} \cdot K_{(\theta_{\mu_1})_{\mu_2}}
\end{align}
where
\begin{equation}
(\theta_{\mu_1})_{\mu_2} = \frac{\tilde k_{\mu_2}^{\theta_{\mu_1}}}{k_{\mu_2}^{\theta_{\mu_1}}}\ .
\end{equation}

By repeating this reasoning and choosing $\mu_j = \lambda_j$ , $j=1,2,\dots,N$, we have the following.

\begin{thm}
\label{thm:model_space_repres_ker_theta_B}
Let $\lambda_1, \dots, \lambda_N \in\DD$, $N\geq 1$ and let $\theta$ be an inner function with $\dim K_\theta > N$.
Let moreover $B = B_{\lambda_1} B_{\lambda_2} \cdots B_{\lambda_N}$ be a finite Blaschke product and
\begin{align}
I_0 &= \theta \qquad \ , \qquad I_n = \frac{\tilde k_{\lambda_n}^{I_{n-1}}}{k_{\lambda_n}^{I_{n-1}}}\ , \\
O_n &= (1-\ol\lambda_n z) k_{\lambda_n}^{I_{n-1}} = 1-\ol{I_{n-1}(\lambda_n)} I_{n-1}\ ,
\end{align}
for $1\leq n \leq N$.

Then:
\begin{enumerate}[(i)]
\item $\ker T_{\ol\theta B} = O_1 \cdots O_N K_{I_N}$;

\item $\ker T_{\ol\theta B} = m K_{I_N}$
where
\begin{equation}
m = \prod_{i=1}^N \frac{O_i}{\sqrt{1 - |I_{i-1}(\lambda_i)|^2}}
\end{equation}
is an isometric multiplier from $K_{I_N}$ onto $\ker T_{\ol\theta B}$;

\item
\begin{equation}
\ker T_{\ol\theta B} = m \frac{k_0^{I_N}}{\sqrt{1-|I_N(0)|^2}} K_{z {(I_N)}_0}
\end{equation}
is Hayashi's representation of $\ker T_{\ol\theta B}$.
\end{enumerate}
\end{thm}

The question of describing the functions $u$ and $\gamma$ in \eqref{eq:Hayashi_representation} for kernels of the form $\ker T_{\ol\theta B}$ where $B$ is a finite Blaschke product with simple zeroes was studied, by a completely different method, in \cite{Kapustin18}.
The simplicity and generality of the results presented here highlight the importance of considering maximal functions and appropriate factorisations of inner factors in our approach.


\section{Toeplitz kernels with generalised Frostman shift symbols}
\label{sec:Tkernels_gen_Frostman_shift_isometric_multipliers}

We now apply the previous results to study a class of Toeplitz kernels that may be seen as a ``perturbation'' of a model space, of the form
\begin{equation}
\label{eq:perturbation_kernel}
\ker T_{\ol\theta - h} \quad , \quad h \in \Hinf, \ ||h||_\infty <1.
\end{equation}

Factorising the symbol as
\begin{equation}
\ol\theta - h = \ol\theta (1 - h\theta)
\end{equation}
where $1-h\theta \in\GG\Hinf$, we see that $\ker T_{\ol\theta - h}$ can be represented in the form
\begin{equation}
\label{eq:model_space_repres_Frostman_shift}
\ker T_{\ol\theta - h} = \frac{1}{1 - h\theta} K_\theta\ .
\end{equation}

One easily sees from \eqref{eq:model_space_repres_Frostman_shift} that the following properties hold.

\begin{prop}
$\ker T_{\ol\theta - h} \neq \zero$ if and only if $\theta$ is non-constant;
$\ker T_{\ol\theta - h}$ is finite dimensional if and only if $\theta$ is a finite Blaschke product.
\end{prop}

\begin{prop}
\label{prop:ker_theta_h}
$\ker T_{\ol\theta - h} = \ker T_{\ol{\theta_h}}$ , where
\begin{equation}
\theta_h = \frac{\theta - \ol h}{1 - h	\theta}
\end{equation}
is a unimodular symbol.
\end{prop}

The function $\theta_h$ is called a \emph{generalised Frostman shift of $\theta$} (\cite{CaCarRoss24}).
Recall that, for any $p\in\CC$, $|p|<1$, one can associate to $\theta$ another inner function, $\theta_p$, defined by
\begin{equation}
\theta_p = \frac{\theta - \ol p}{1 - p \theta} \ ,
\end{equation}
called a {\em Frostman shift of $\theta$}.
Thus Toeplitz kernels of the form \eqref{eq:perturbation_kernel} can also be seen as a generalisation of model spaces defined by Frostman shifts.
Those Toeplitz kernels are model spaces if and only if $h$ is a constant, $h=p \in\CC$, $|p|<1$.
In that case, by Theorem \ref{thm:Crofoot_isom_multiplier}, $K_\theta$ can be isometrically multiplied onto $K_{\theta_p}$:
\begin{equation}
\ker T_{\ol\theta -p} = \ker T_{\ol{\theta_p}} = K_{\theta_p} = \frac{\sqrt{1 - |p|^2}}{1-p\theta} K_\theta
\end{equation}
where the isometric multiplier
\begin{equation}
m_p^\theta = \frac{\sqrt{1 - |p|^2}}{1-p\theta}
\end{equation}
is a constant multiple of the multiplier in \eqref{eq:model_space_repres_Frostman_shift} for $h=p$.
A natural question then is to characterise the class of functions $h\in\Hinf$, with $||h||_\infty <1$, such that a constant multiple of that multiplier, i.e., $\frac{C}{1-h\theta}$ with $C\in\CC$, is an isometric multiplier from $K_\theta$ onto $\ker T_{\ol{\theta_h}}$.

By \cite[Proposition 5.8]{CaCarRoss24}, this happens if and only if the truncated Toeplitz operator
\begin{equation}
A^\theta_{1-\frac{|C|^2}{|1-h\theta|^2}}  = P_\theta \left( 1-\frac{|C|^2}{|1-h\theta|^2} \right) P_\theta \quad : K_\theta \to K_\theta
\end{equation}
is the zero operator on $K_\theta$ .
This, in its turn, is equivalent to
\begin{equation}
\label{eq:zero_symbol}
1-\frac{|C|^2}{|1-h\theta|^2} \in \ol{\theta\Hp} + \theta \Hp
\end{equation}
(\cite{Sa07}).
Now we can write
\begin{equation}
1-\frac{|C|^2}{|1-h\theta|^2} = \frac{1 - |C|^2 - |h|^2}{|1-h\theta|^2} + 
\annot{\frac{\theta h}{1-h\theta}}{\in \theta \Hp} + 
\annot{\frac{\ol\theta \ol h}{1-\ol h \ol\theta}}{\in\ol\theta \ol\Hp}\ ,
\end{equation}
so \eqref{eq:zero_symbol} is reduced to imposing that
\begin{equation}
\label{eq:Frostman_shift_mult_isometric_condition}
\frac{1 - |C|^2 - |h|^2}{|1-h\theta|^2} \in \ol{\theta\Hp} + \theta \Hp\ .
\end{equation}

We have thus the following necessary and sufficient condition.

\begin{prop}[{\cite{CaCarRoss24}}]
The multiplier $\frac{C}{1-h\theta}$ , where $C\in\CC$, from $K_\theta$ onto $\ker T_{\ol{\theta_h}}$ is an isometric multiplier if and only if \eqref{eq:Frostman_shift_mult_isometric_condition} holds.
\end{prop}

We thus get an isometric model space representation for $\ker T_{\ol\theta -p\alpha}$, with $p\in\CC$, $|p|<1$ and $\alpha$ inner.

\begin{cor}
\label{cor:Frostman_shift_h=p*alpha}
If $h=p\alpha$ where $p\in\CC$, $|p|<1$ and $\alpha$ is an inner function, then $m_p^{\alpha\theta} = \frac{\sqrt{1-|p|^2}}{1-p\alpha\theta}$ is an isometric multiplier from $K_\theta$ onto $\ker T_{\ol\theta -p\alpha} = \ker T_{\ol{\theta_{p\alpha}}}$.
\end{cor}

\begin{proof}
\eqref{eq:Frostman_shift_mult_isometric_condition} obviously holds if $1-|C|^2-|h|^2 = 0$, which is equivalent to $|C|<1$ and $|h|/ \sqrt{1-|C|^2} =1$.
This means that $h/ \sqrt{1-|C|^2} = \alpha$ is an inner function, i.e., $h=p\alpha$ with $|p|=\sqrt{1-|C|^2} <1$.
\end{proof}

Another approach to obtaining a model space representation of the form $\ker T_{\ol\theta -h} = w K_\beta$, where $w$ is an isometric multiplier from $K_\beta$ onto $\ker T_{\ol\theta -h}$, consists in using Theorem \ref{thm:Crofoot_isom_multiplier} and Proposition \ref{prop:multiplier_between_subkernels}, when $\ker T_{\ol\theta - h}$ is contained in a model space as follows.

\begin{prop}
\label{prop:ker_theta_h_contained_model_space}
Let $h\in\Hinf$ with $||h||_\infty <1$.
Then $\ker T_{\ol\theta -h}$ is contained in a model space if and only if
\begin{equation}
\label{eq:K_alpha_infty}
h\in K_\alpha^\infty := \{ f\in\Hinf : \ol\alpha f \in \ol\Hinf \}
\end{equation}
for some inner function $\alpha$.
In that case, defining
\begin{equation}
\label{eq:model_space_containing_Frostman_shift}
\gamma = \frac{\theta \alpha - \alpha \ol h}{1 - h\theta} = \theta_h \alpha \ ,
\end{equation}
we have that $\gamma$ is an inner function and
\begin{equation}
\label{eq:Frostman_shift_equal_gamma_alpha}
\ker T_{\ol\theta -h} = \ker T_{\ol\gamma \alpha} \subset K_\gamma \ .
\end{equation}
\end{prop}

\begin{proof}
The Toeplitz kernels containing $\ker T_{\ol\theta - h} = \ker T_{\ol{\theta_h}}$ are those of the form $\ker T_{\ol{\theta_h}\ol\alpha}$ where $\alpha$ is an inner function (Corollary \ref{cor:subkernel_galpha}).
We have that $\ker T_{\ol{\theta_h}\ol\alpha}$ is a model space if and only if
\begin{equation}
\ol{\theta_h}\ol\alpha \in\ol\Hinf \iff (\ol\theta - h) \ol\alpha \in\ol\Hinf \iff h\ol\alpha \in\ol\Hinf \ .
\end{equation}
In that case, since $|\theta_h \alpha|=1$, we have that $\gamma = \theta_h \alpha$ is an inner function and \eqref{eq:Frostman_shift_equal_gamma_alpha} holds.
\end{proof}

Note that if $\alpha$, in Proposition \ref{prop:ker_theta_h_contained_model_space}, is a finite Blaschke product, then we can use the results from Section \ref{sec:multipliers_subkernels} to obtain an isometric model space representation for $\ker T_{\ol\theta -h}$.
For a general inner function, we have the results that follow.

\begin{prop}
\label{prop:isometric_multiplier_Frostman_shift}
Let $h$ satisfy \eqref{eq:K_alpha_infty} with $||h||_\infty < 1$, and let $\gamma$ be defined by \eqref{eq:model_space_containing_Frostman_shift}.
For any $p\in\CC$, $|p|<1$, define
\begin{equation}
m_p^\gamma = \frac{\sqrt{1-|p|^2}}{1-p\gamma}
\quad , \quad
\gamma_p = \frac{\gamma - \ol p}{1 - p\gamma}\ .
\end{equation}
Then $m_p^\gamma$ is an isometric multiplier from $\ker T_{\ol\theta -h}$ onto $\ker T_{\ol{\gamma_p} \alpha}$.
\end{prop}

\begin{proof}
(Figure \ref{fig:diagram})
We have that $m_p^\gamma$ is an isometric multiplier from $K_\gamma$ onto $K_{\gamma_p}$, so from \eqref{eq:Frostman_shift_equal_gamma_alpha} and Corollary \ref{cor:multiplier_between_subkernels_v2} we have that
\begin{equation}
m_p^\gamma \ker T_{\ol\theta -h} = m_p^\gamma \ker T_{\ol\gamma \alpha} = \ker T_{\ol{\gamma_p} \alpha}\ ,
\end{equation}
where $m_p^\gamma$ is an isometric multiplier because $\ker T_{\ol\theta -h} \subset K_\gamma$ and $\ker T_{\ol{\gamma_p} \alpha} \subset K_{\gamma_p}$.
\end{proof}

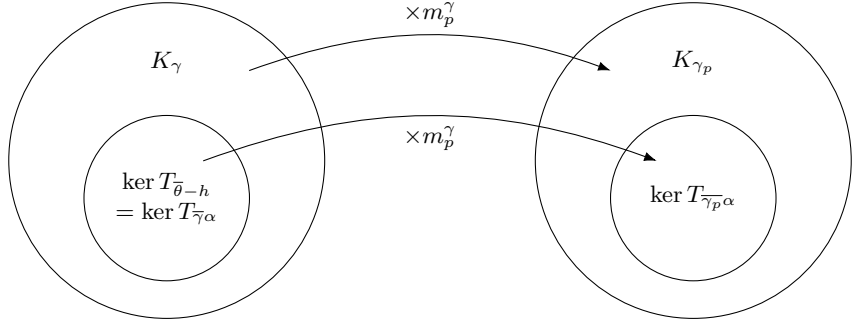
\begin{figure}[H]
\centering
\begin{tikzpicture}[
  >=Latex,
  font=\small,
  big/.style={circle, draw, minimum size=4.2cm},
  small/.style={circle, draw, minimum size=2.2cm},
]

\node[big] (L) at (0,0) {};
\node[big] (R) at (7,0) {};

\node[small] (Ls) at (0,-0.5) {$\begin{array}{c}
\ker T_{\overline{\theta}-h} \\
= \ker T_{\overline{\gamma}\alpha}
\end{array}$};
\node[small] (Rs) at (7,-0.5) {$\ker T_{\overline{\gamma_p}\alpha}$};

\node at (L.north) [yshift=-0.8cm] {$K_\gamma$};
\node at (R.north) [yshift=-0.8cm] {$K_{\gamma_p}$};

\draw[->, bend left=20, shorten >=4pt, shorten <=4pt]
  ([shift={(50:1.5cm)}]L.center) to
  node[above] {$\times m_p^\gamma$}
  ([shift={(130:1.5cm)}]R.center);

\draw[->, bend left=20, shorten >=3pt, shorten <=3pt]
  ([shift={(50:0.6cm)}]Ls.center) to
  node[below] {$\times m_p^\gamma$}
  ([shift={(130:0.6cm)}]Rs.center);

\end{tikzpicture}
\caption{$\ker T_{\ol\theta -h}$ contained in $K_\gamma$, and the multiplier $m_p^\gamma$ acting on both spaces}
\label{fig:diagram}
\end{figure}

\begin{cor}
\label{cor:Frostman_shift_isom_model_space_repres}
With the same assumptions as in Proposition \ref{prop:isometric_multiplier_Frostman_shift}, if $\alpha \preceq \gamma_p$ then
\begin{equation}
\ker T_{\ol\theta -h} = (m_p^\gamma)^{-1} K_{\gamma_p / \alpha}
\end{equation}
is an isometric model space representation of $\ker T_{\ol\theta -h}$.
\end{cor}

We have the following conditions for $\alpha$ to divide $\gamma_p$.

\begin{prop}
\label{prop:alpha_divides_gamma_p}
Let $p\in\CC$, $|p|<1$, $h\in K_\alpha^\infty$ for some inner function $\alpha$ and $\gamma$ be defined by \eqref{eq:model_space_containing_Frostman_shift}.
Then $\alpha \preceq \gamma_p$ if and only if
\begin{equation}
\label{eq:gamma_p_alpha_is_model_space}
\ol h + \ol p\, \ol\alpha \, (1 - h\theta) \in \Hinf.
\end{equation}
\end{prop}

\begin{proof}
Since $1-p\gamma \in\GG\Hinf$, we have that
\begin{align}
\alpha \preceq \gamma_p
	& \iff 
	\gamma_p \ol\alpha \in\Hinf	 \iff (\gamma - \ol p) \ol\alpha \in\Hinf\\
& \iff 
	\theta - \ol h - \ol p\, \ol\alpha + h\,\ol\alpha\, \ol p\, \theta \in\Hinf \\
& \iff 
	\ol h + \ol p\, \ol\alpha \, (1 - h\theta) \in \Hinf \ .
\end{align}
\end{proof}

As a consequence of Proposition \ref{prop:alpha_divides_gamma_p} and Corollary \ref{cor:Frostman_shift_isom_model_space_repres}, we have the following.

\begin{cor}
\label{cor:constant_h+palpha}
If $\alpha \preceq \theta$ and $h = C - p\alpha$, with $C\in\CC$ and $||h||_\infty < 1$, then $\alpha \preceq \gamma_p$ and
\begin{equation}
\ker T_{\ol\theta -h} = (m_p^\gamma)^{-1} K_{\gamma_p \ol\alpha}
\end{equation}
is an isometric model space representation of $\ker T_{\ol\theta -h}$.
\end{cor}

\begin{proof}
We have that $\ol\alpha h = C \ol\alpha - p \in\ol\Hinf$, thus $h\in K_\alpha^\infty$.
Moreover,
\begin{equation}
\ol h + \ol p \ol\alpha (1-h\theta) = \ol C - C \ol p \theta \ol\alpha + |p|^2 \theta \in\Hinf,
\end{equation}
therefore, from Proposition \ref{prop:alpha_divides_gamma_p}, we have that $\alpha \preceq \gamma_p$, and the result follows from Corollary \ref{cor:Frostman_shift_isom_model_space_repres}.
\end{proof}

\begin{rem}
\label{rem:ker_theta_palpha}
The case where $h=p\alpha$, considered in Corollary \ref{cor:Frostman_shift_h=p*alpha}, is a particular case where $h\in K_\alpha^\infty$.
We see that, in that case, $\gamma = \frac{\theta \alpha - \ol p}{1 - p \theta \alpha} = (\theta\alpha)_p$ is a Frostman shift of $\theta\alpha$.
From Propositions \ref{prop:ker_theta_h} and \ref{prop:ker_theta_h_contained_model_space} and Theorem \ref{thm:Crofoot_isom_multiplier}, we have
\begin{equation}
\label{eq:rem:ker_theta_palpha}
\ker T_{\ol\theta - p\alpha} = \ker T_{\ol{\theta_{p\alpha}}} \subset 
K_\gamma = K_{(\theta\alpha)_p} = \frac{\sqrt{1-|p|^2}}{1-p\theta\alpha} K_{\theta\alpha} = m_p^{\theta\alpha} K_{\theta\alpha}\ ,
\end{equation}
where $m_p^{\theta\alpha}$ is an isometric multiplier from $K_{\theta\alpha}$ onto $K_{(\theta\alpha)_p}$.
Since $\ker T_{\ol\theta - p\alpha} = m_p^{\theta\alpha} K_\theta$, with $K_\theta \subset K_{\theta\alpha}$, one can also derive the result of Corollary \ref{cor:Frostman_shift_h=p*alpha} from \eqref{eq:rem:ker_theta_palpha}.
\end{rem}

\begin{ex}
Let $\theta$ be an inner function and consider the affine perturbation $h=a+bz$, with $a,b\in\CC$ such that $|a|+|b|<1$, so that $||h||_\infty <1$.
Then, with $\alpha = z$, since $\ol z h = a \ol z + b \in \ol\Hinf$, we have that $h\in K_z^\infty$, and from Proposition \ref{prop:ker_theta_h_contained_model_space}
\begin{equation}
\ker T_{\ol \theta - a - bz} = \ker T_{\ol\gamma z} \subset K_\gamma \,
\end{equation}
with $\gamma = \theta_h z = \frac{\theta z - \ol a z - \ol b}{1 - (a+bz) \theta}$.

On the one hand, from Theorem \ref{thm:model_space_repres_ker_theta_Blambda},
with
\begin{equation}
\beta = \frac{\tilde k^\gamma_0}{k^\gamma_0}
\quad , \quad
w = \frac{1 - \ol{\gamma(0)} \gamma}{\sqrt{1 - |\gamma(0)|^2}}\ ,
\end{equation}
\begin{equation} \label{eq:ex:affine_gamma_z}
\ker T_{\ol \theta - a - bz}
= \ker T_{\ol\gamma z}
= w K_\beta
= \left( m_{p}^\gamma \right)^{-1} K_\beta
\quad , \quad
\text{with } p = \ol{\gamma(0)} \ ,
\end{equation}
is an isometric model space representation of $\ker T_{\ol \theta - a - bz}$.

On the other hand, $z \preceq \gamma_p \iff \gamma_p(0) = 0 \iff p = \ol{\gamma(0)}$.
So, by Corollary \ref{cor:Frostman_shift_isom_model_space_repres} with $p = \ol{\gamma(0)}$,
\begin{equation} \label{eq:ex:affine_frostman_shift}
\ker T_{\ol \theta - a - bz}
= \left( m_{p}^\gamma \right)^{-1} K_{\gamma_p \ol z}
= \left( m_{p}^\gamma \right)^{-1} K_\beta \ ,
\end{equation}
and we obtain the same isometric model space representation of $\ker T_{\ol\theta -a-bz}$.
\end{ex}

\begin{ex}
Consider the singular inner functions $e_k(z) := \exp \left( k \frac{z+1}{z-1} \right)$, with $k > 0$.
Let $\theta = e_1$ and $\alpha = e_k$, for some $0 < k \leq 1$, and let $h=a+b e_k$, with $a,b\in\CC$ such that $|a|+|b|<1$, so that $||h||_\infty <1$.
Then, $\ol{e_k} h = a \ol{e_k} + b \in \ol\Hinf$, and by Proposition \ref{prop:ker_theta_h_contained_model_space}, defining
\begin{equation}
\gamma = (e_1)_h e_k = \frac{e_{1+k} - \ol a e_k - \ol b}{1 - a e_1 - b e_{1+k}} \ ,
\end{equation}
we have
\begin{equation}
\ker T_{\ol{e_1} - a - b e_k} = \ker T_{\ol\gamma e_k} \subset k_\gamma \ .
\end{equation}
From Proposition \ref{prop:isometric_multiplier_Frostman_shift}, $\left( m_{p}^\gamma \right)^{-1}$ is an isometric multiplier from $\ker T_{\ol{\gamma_p} e_k}$ onto $\ker T_{\ol{e_1} - a - b e_k}$, for any $p\in\CC$, $|p|<1$.
And from Proposition \ref{prop:alpha_divides_gamma_p}, we have that $e_k \preceq \gamma_p$, i.e. $\ker T_{\ol{\gamma_p} e_k}$ is a model space, if and only if
\begin{align}
\ol h + \ol p\, \ol\alpha \, (1 - h\theta) \in \Hinf 
& \iff \ol a + \ol b \ol{e_k} (1 - a e_1 - b e_{k+1}) \in\Hinf \\
& \iff (\ol b + \ol p) \ol{e_k} \in \Hinf \\
& \iff p = -b \ .
\end{align}
Thus, choosing $p=-b$, from Corollary \ref{cor:Frostman_shift_isom_model_space_repres}, we have that
\begin{equation}
\ker T_{\ol{e_1} - a - b e_k}
= \left( m_{-b}^\gamma \right)^{-1} K_{\gamma_{{}_{-b}} \ol{e_k}}
\end{equation}
is a model space representation of $\ker T_{\ol{e_1} - a - b e_k}$, where
\begin{equation}
\gamma_{{}_{-b}}\ \ol{e_k} = \frac{\gamma + \ol b}{1 + b\gamma} \ol{e_k} = \frac{e_1 - g}{1 - \ol g e_1}
\quad , \text{ with} \quad
g = \frac{\ol a + a \ol b e_{1-k}}{1 - |b|^2} \ ,
\end{equation}
and
\begin{equation}
\left( m_{-b}^\gamma \right)^{-1} = \frac{1+b\gamma}{\sqrt{1-|b|^2}} = \sqrt{1 - |b|^2} \frac{1 - \ol g e_1}{1 - a e_1 - b e_{1+k}} \ .
\end{equation}
\end{ex}

\section*{Acknowledgements}
Research partially funded by Fundação para a Ciência e Tecnologia (FCT), Portugal, through grant No. UID/4459/2025.
Also, the work of the second author was supported by Fundação para a Ciência e Tecnologia (FCT), Portugal, through CAMGSD PhD fellowship UI/BD/153700/2022.

\section*{Author contributions}
All authors contributed to the research in this manu-
script, and they read and approved the final manuscript.

\section*{Data Availibility}
Data sharing is not applicable to this article as no datasets
were generated or analysed during the current study.

\section*{Competing Interests}
The authors have no relevant financial or non-financial interests to disclose.

\bibliographystyle{plain}
\bibliography{references}

\end{document}